\documentclass{siamltex}
\usepackage{a4wide,amssymb,graphics,graphicx,textcomp,color}
\usepackage{enumerate,textcomp,multirow,amsmath}
\newtheorem{remark}[theorem]{Remark}
\newtheorem{example}[theorem]{Example}
\numberwithin{equation}{section}
\numberwithin{table}{section}
\numberwithin{figure}{section}
\newcommand {\mat}  [1] {\left[\begin{array}{#1}}
\newcommand {\rix}      {\end{array}\right]}
\newcommand{\eproof}{\space
    {\ \vbox{\hrule\hbox{\vrule height1.3ex\hskip0.8ex\vrule}\hrule}}\par}
\newif\ifMDlatex
\catcode`@=11     
\def\MD@us#1{\csname#1style\endcsname}
\def\MD@uf#1{\csname#1font\endcsname}
\def\MD@t{text}
\def\MD@s{script}
\def\MD@ss{scriptscript}
\newdimen\MD@unit
\MD@unit\p@
\def\MD@changestyle#1{
  \relax\MD@unit0.1\fontdimen6\MD@uf{#1}0
  \everymath\expandafter{\the\everymath\MD@us{#1}}
}
\def\MD@dot{$\m@th\ldotp$}
\def\MD@palette#1{\mathchoice{#1\MD@t}{#1\MD@t}{#1\MD@s}{#1\MD@ss}}
\def\MD@ddots#1{{\MD@changestyle{#1}%
  \mkern1mu\raise7\MD@unit\vbox{\kern7\MD@unit\hbox{\MD@dot}}%
  \mkern2mu\raise4\MD@unit\hbox{\MD@dot}%
  \mkern2mu\raise \MD@unit\hbox{\MD@dot}\mkern1mu}}%
\def\MD@iddots#1{{\MD@changestyle{#1}%
  \mkern1mu\raise \MD@unit\hbox{\MD@dot}%
  \mkern2mu\raise4\MD@unit\hbox{\MD@dot}%
  \mkern2mu\raise7\MD@unit\vbox{\kern7\MD@unit\hbox{\MD@dot}}}}%
\def\MD@vdots#1{\vbox{\MD@changestyle{#1}%
    \baselineskip4\MD@unit\lineskiplimit\z@
    \kern6\MD@unit\hbox{\MD@dot}\hbox{\MD@dot}\hbox{\MD@dot}}}%
\ifMDlatex
  \DeclareRobustCommand\ddots{\mathinner{\MD@palette\MD@ddots}}%
  \DeclareRobustCommand\iddots{\mathinner{\MD@palette\MD@iddots}}%
  \DeclareRobustCommand\vdots{\mathinner{\MD@palette\MD@vdots}}%
\else
  \def\ddots{\mathinner{\MD@palette\MD@ddots}}%
  \def\iddots{\mathinner{\MD@palette\MD@iddots}}%
  \def\vdots{\mathinner{\MD@palette\MD@vdots}}%
\fi
\catcode`@=12    
\newcommand{\adots}{\iddots}

\begin{document}

\title{Low-rank perturbation of symplectic matrices}

\author{Christian Mehl\thanks{Institut f\"ur Mathematik, Technische Universit\"at Berlin, Sekretariat MA 4-4, Fasanenstra{\ss}e 89,
10623 Berlin, Germany ({\tt  mehl@math.tu-berlin.de})}
}

\maketitle

\begin{abstract}
In this paper the generic behavior of the Jordan structure of symplectic matrices under low-rank structure-preserving transformations
is considered. So far, only the case of rank-one perturbations was covered in the literature, and this result is now extended to
the case of arbitrary rank $k$. As a key step toward the main result, it is shown that all symplectic matrices that differ from the
identity by rank $k$ can be represented as a product of $k$ symplectic reflectors unless they are similar to a diagonalizable matrix
having only the eigenvalues $\pm 1$.
\end{abstract}

\section{Introduction}

In \cite{HorM94} the change in Jordan structure of complex matrices and compact operators was studied under perturbations that are low in rank,
but not necessarily small in norm. The main result was later rediscovered independently in \cite{MorD03} and \cite{Sav03,Sav04} for the matrix case and initiated a new branch of
research on low-rank perturbations of matrices, matrix pencils and matrix polynomials, sometimes including the real or quaternionic cases, see, e.g., \cite{DeD07,DeD09,DeD16,DeMD08,MMRR4,MR17}. While some results consider the possible behavior under arbitrary perturbations, see also
\cite{BarR19,BarR20} and for generalizations to linear operators and linear relations also \cite{BehLMT15,LebMPTW21}, the main focus
in the aforementioned papers is on generic perturbations, because in that case there is a unique behavior in the change of the Jordan
structure that can be described as follows: if in the Jordan canonical form of the considered matrix there are $g$ Jordan blocks associated
with a particular eigenvalue, then after a generic perturbation of rank $k$, the $k$ largest Jordan blocks will disappear from
the Jordan canonical form of the perturbed matrix while the smallest $g-k$ Jordan blocks will remain. 
(The considered eigenvalue will not be an eigenvalue of the perturbed matrix if $g\leq k$.)

Starting with \cite{MMRR2}, the corresponding problem for matrices and matrix pencils with symmetry structures has been considered, 
see, e.g., \cite{BatMRR16,DeMM17,FouGJR13,Jan12}. At first, the main focus was on the rank-one case, but it was later shown in \cite{BatMRR16} how the results can be generalized to structure-preserving perturbations of arbitrary rank.
For many of these structure, it has been observed that in some situations the behavior under generic structure-preserving perturbations  
differs considerably from the behavior under generic perturbations that ignore the structure.

The case of structure-preserving rank-one perturbations of real and complex symplectic matrices has been considered in \cite{MMRR5}, 
but in contrast to other structures, the corresponding results have not yet been generalized to structure-preserving perturbations 
of higher rank. The reason for this gap in the literature 
is the fact that structure-preserving perturbations of symplectic matrices have to be considered as multiplicative
rather than additive perturbations, because of the group structure of the set of symplectic matrices. This is in contrast to the main result 
in \cite{BatMRR16} that explicitly exploits the viewpoint of additive perturbations by representing a structure-preserving perturbation 
of rank $k$ as a sum of $k$ rank-one perturbations.

Thus, in order to develop a theory of structure-preserving rank-$k$ perturbations of symplectic matrices in the spirit of \cite{BatMRR16}, 
it is necessary to represent a symplectic rank-$k$ perturbation as a product of exactly $k$ symplectic rank-one perturbations.
Here, a \emph{symplectic rank-$k$ perturbation} is understood as a symplectic matrix that differs from the identity by a rank-$k$ matrix.
It was shown in \cite{MMRR5} that any symplectic rank-one perturbation can be represented as a \emph{symplectic reflector}, see
Definition~\ref{def:sr} for the exact introduction of this term. Thus, the question raised above can be reformulated as the problem
whether each symplectic matrix that differs from the identity by a rank-$k$ matrix can be written as a product of $k$ symplectic reflectors.
As it turns out, the answer is negative in general, but in 1955 Dieudonn\'{e} (\cite{Die55}, see also \cite{OMe78}) gave a necessary and
sufficient condition when the answer is affirmative. 
The results and proofs given in \cite{Die55} and \cite{OMe78} cover arbitrary fields of characteristic different from
two and are aimed at a readership specialized in algebra. However, it seems that they are not easily comprehensible for
a reader interested in matrix perturbation theory. Therefore, an alternative equivalent condition is provided in this paper
accompanied with a proof based on canonical forms. The author believes that the reformulated result and the alternative proof highlight 
the connection of the result to matrix perturbation theory and in particular allow us to explain why the equivalent condition is exactly 
how it is, thereby leading to a better understanding of the underlying theory.
Based on this result, it is then shown that the set of symplectic rank-$k$ perturbations
that can be represented as a product of $k$ reflectors is generic. This finally allows to generalize the result on the generic behavior of
symplectic matrices under structure-preserving rank-one perturbations to the case of structure-preserving perturbations of rank-$k$.

The remainder of the paper is organized as follows. Section 2 provides preliminary results and reviews results from the literature that
are needed in the following. Section 3 considers products of symplectic reflectors and presents an alternative proof of Dieudonn\'{e}'s
result while Section~4 contains the main perturbation theoretical results of the paper concerning the generic behavior of the Jordan
structure of a symplectic matrix under structure-preserving low-rank perturbations. 

\section{Preliminaries}

In this section we will introduce the relevant notation and definitions and review the canonical form for symplectic matrices as
well as the result on generic structure-preserving rank-one perturbations from \cite{MMRR5}.
$\mathbb F^{n,m}$ stands for the set of $n\times m$ matrices over the field $\mathbb F$, where $\mathbb F\in\{\mathbb R,\mathbb C\}$.
If the matrix $M\in\mathbb F^{m,m}$ is invertible, then we use $M^{-\top}$ as an abbreviation for $(M^\top)^{-1}=(M^{-1})^\top$.
A set $\mathcal A\subseteq\mathbb F^m$ is called \emph{algebraic} if it is the set of common zeros of finitely many 
polynomials over $\mathbb F$ in $m$ variables. An algebraic set $\mathcal A\subseteq\mathbb F^m$ is called \emph{proper} if
$\mathcal A\neq\mathbb F^m$. (This is equivalent to saying that the finitely many polynomials from the definition of an algebraic
set are not all zero.) Finally, a set $\mathcal G\subseteq\mathbb F^m$ is called \emph{generic} if its complement is contained in
a proper algebraic set and a set $\mathcal G\subseteq\mathbb F^{m,n}$ is called \emph{generic} if it is generic when canonically
identified with a subset of $\mathbb F^{mn}$. Next we introduce symplectic matrices where we use a slightly more general notion.

\begin{definition}
Let $J\in\mathbb F^{2n,2n}$ be an invertible skew-symmetric matrix. Then a matrix $S\in\mathbb F^{2n,2n}$ is called \emph{$J$-symplectic} if $S^\top JS=J$.
\end{definition}

If $J$ is the particular matrix
\begin{equation}\label{eq:J}
J=\mat{cc} 0&I_n\\ -I_n&0\rix
\end{equation}
then a $J$-symplectic matrix is simply called \emph{symplectic}. In this paper, we will however consider the more general notion of
$J$-symplectic matrices, because this will allow us to transform the pair $(S,J)$ to a canonical form from which the Jordan 
structure of $S$ can be easily read off. This form is obtained under the following equivalence relation on
the set of $m\times m$ matrix pairs:
Two pairs of matrices $(S,J)$ and $(\widetilde S,\widetilde J)\in (\mathbb F^{m,m})^2$
are called \emph{equivalent} if there exists a nonsingular matrix $P$ such that
\[
\widetilde S=P^{-1}SP\quad\mbox{and}\quad \widetilde J=P^\top JP.
\]
It can be easily verified that for the case $m=2n$ and $J,\widetilde J$ being skew-symmetric and invertible we have that
$\widetilde S$ is $\widetilde J$-symplectic if and only if $S$ is $J$-symplectic, i.e. the symplectic structure is
preserved under such equivalence transformations. This becomes even more clear by interpreting $P$ as a transformation matrix corresponding 
to a change of basis in a space with an alternating inner product. While for the linear map represented by $S$ such a basis change will result
in a similarity transformation on $S$, the corresponding effect on the matrix $J$ representing the alternating inner product in the given 
space will be a congruence transformation. Thus, the equivalence transformation only changes the particular
representations of the matrices, but not their relation within the geometry of the space. 

A $J$-symplectic matrix $S$ is called
\emph{$J$-indecomposable} if it is not equivalent to a direct sum of two nontrivial symplectic matrices
of smaller size. To be precise, if $S$ is $J$-indecomposable and $P$ is nonsingular such that
\[
P^{-1}SP=S_1\oplus S_2\quad\mbox{and}\quad P^\top JP=J_1\oplus J_2,
\]
where $S_i$ and $J_i$ have the same size for $i=1,2$, then either $(S_1,J_1)$ or $(S_2,J_2)$ is void.

The following canonical forms were presented in \cite{Mehl1} and \cite{Mehl2}, respectively, and they have been
slightly adapted here to match the notation used in this paper. In the following, $T(a_1,\dots,a_m)$ denotes the upper
triangular $m\times m$ Toeplitz matrix whose first row is given by $\mat{ccc}a_1&\dots&a_m\rix$,
while $\Sigma_m$ is the $m\times m$ anti-diagonal matrix with anti-diagonal entries that
alternate between $1$ and $-1$ starting with $1$ in the first row, i.e.,
\[
T(a_1,\dots,a_m)=\mat{cccc} a_1&a_2&\dots&a_m\\ 0&\ddots&\ddots&\vdots\\
\vdots & \ddots &\ddots & a_2\\ 0&\dots&0&a_1\rix\quad\mbox{and}\quad
\Sigma_m=\mat{cccc}0&\dots&0&1\\[-1ex] \vdots & \adots & -1 &0\\[-1ex] 0& \adots &\adots &\vdots\\
\!\!(-1)^{m-1}\!\!\!\! & 0&\dots&0\rix.
\]
Moreover, $J_m(\lambda_0)$ denotes the standard upper triangular Jordan block of size $m\times m$ associated with
the eigenvalue $\lambda_0$ and $J_m(\alpha,\beta)$ stands for the upper triangular real Jordan block of size $2m\times 2m$
associated with a pair of conjugate complex eigenvalues $\alpha\pm i\beta$.

\begin{theorem}[Canonical form for complex $J$-symplectic matrices]\label{cbilkunithm}
Let $J\in\mathbb C^{2n, 2n}$ be skew-symmetric and invertible and $S\in\mathbb C^{2n, 2n}$ be
$J$-symplectic. Then there exists a nonsingular matrix $P\in\mathbb C^{2n, 2n}$ such that
\begin{equation}\label{cbilkuni}
P^{-1}SP=S_1\oplus\cdots\oplus S_p,\quad
P^\top JP=J_1\oplus\cdots\oplus J_p,
\end{equation}
where $S_j$ is $J_j$-indecomposable and $S_j$ and $J_j$ have one of the following forms for $j=1,\dots,p$:
\begin{itemize}
\item[i)] even-sized blocks associated with $\lambda_j=\delta=\pm 1$, where $m_j\in\mathbb N$ is even:
\begin{equation}\label{cbilkuni1}
S_j=T(\delta,1,r_2,\dots,r_{m_j-1}),\quad J_j=\Sigma_{m_j},
\end{equation}
Moreover, $r_{k}=0$ for odd $k$ and the parameters $r_k$
for even $k$ are real and uniquely determined by the recursive formula
\begin{equation}\label{recformulaorthskew}
r_2=\frac{1}{2}\delta,\quad
r_{k}=-\frac{1}{2}\delta\left(\sum_{\nu=1}^{\frac{k}{2}-1}r_{2\cdot \nu}r_{2\cdot(\frac{k}{2}-\nu)}\right),
\quad 4\leq k\leq m_j;
\end{equation}
\item[ii)] paired blocks associated with $\lambda_j=\pm 1$, where $m_j\in\mathbb N$ is odd:
\begin{equation}\label{cbilkuni2}
S_j=\mat{cc}\mathcal J_{m_j}(\lambda_j)&0\\ 0& \big(\mathcal J_{m_j}(\lambda_j)\big)^{-\top}\rix,\quad
J_j=\mat{cc} 0 & I_{m_j}\\ -I_{m_j} & 0\rix;
\end{equation}
\item[iii)] blocks associated with a pair $(\lambda_j,\lambda_j^{-1})\in\mathbb C\times\mathbb C$,
satisfying $\operatorname{Re}(\lambda_j)>\operatorname{Re}(\lambda_j^{-1})$, or
$\operatorname{Im}(\lambda_j)>\operatorname{Im}(\lambda_j^{-1})$ if $\operatorname{Re}(\lambda_j)=\operatorname{Re}(\lambda_j^{-1})$,
where $m_j\in\mathbb N$:
\begin{equation}\label{cbilkuni3}
S_j=\mat{cc}\mathcal J_{m_j}(\lambda_j)&0\\ 0& \big(\mathcal J_{m_j}(\lambda_j)\big)^{-\top}\rix,\quad
J_j=\mat{cc} 0 & I_{m_j}\\ -I_{m_j} & 0\rix.
\end{equation}
\end{itemize}
The values $\lambda_j$ and $m_j$ may be different for different values of $j$.
Moreover, the form~\eqref{cbilkuni} is unique up to the permutation of blocks.
\end{theorem}

There is a corresponding real version of Theorem~\ref{cbilkunithm}. For its formulation, we introduce the
reverse identity $R_m$ which is the $m\times m$ matrix with ones on the anti-diagonal and zeros elsewhere.

\begin{theorem}[Canonical form for real symplectic matrices]\label{symmunitthm}
Let $J\in\mathbb R^{2n, 2n}$ be skew-symmetric and invertible and  $S\in\mathbb R^{2n,2n}$ be $J$-symplectic.
Then there exists a nonsingular matrix $P\in\mathbb R^{2n,2n}$ such that
\begin{equation}\label{symmunit}
P^{-1}SP=S_1\oplus\cdots\oplus S_p,\quad P^\top JP=J_1\oplus\cdots\oplus J_p,
\end{equation}
where $S_j$ is $J_j$-indecomposable and where $S_j$ and $J_j$ have one of the
following forms for $j=1,\dots,p$:
\begin{itemize}
\item[i)] even-sized blocks associated with $\lambda_j=\delta=\pm 1$, where $m_j\in\mathbb N$ is even:
\begin{equation}\label{symmunit1}
S_j=T(\delta,1,r_2,\dots,r_{m_j-1}),\quad J_j=\sigma_j\Sigma_{m_j}.
\end{equation}
Moreover, $\sigma_j=\pm 1$, $r_{k}=0$ for odd $k$ and the parameters $r_k$
for even $k$ are real and uniquely determined by the recursive formula~\eqref{recformulaorthskew}.
\item[ii)] paired blocks associated with $\lambda_j=\pm 1$, where $m_j\in\mathbb N$ is odd:
\begin{equation}\label{symmunit2}
S_j=\mat{cc}\mathcal J_{m_j}(\lambda_j)&0\\ 0& \big(\mathcal J_{m_j}(\lambda_j)\big)^{-\top}\rix,\quad
J_j=\mat{cc} 0 & I_{m_j}\\ - I_{m_j} & 0\rix;
\end{equation}
\item[iii)] blocks associated with a pair $(\lambda_j,\lambda_j^{-1})\in\mathbb R\times\mathbb R$,
where $\lambda_j>\lambda_j^{-1}$ and $m_j\in\mathbb N$:
\begin{equation}\label{symmunit3}
S_j=\mat{cc}\mathcal J_{m_j}(\lambda_j)&0\\ 0& \big(\mathcal J_{m_j}(\lambda_j)\big)^{-\top}\rix,\quad
J_j=\mat{cc} 0 & I_{m_j}\\ -I_{m_j} & 0\rix.
\end{equation}
\item[iv)] blocks associated with a pair $(\lambda_j,\overline{\lambda_j})$ of
nonreal unimodular eigenvalues:
\begin{equation}\label{symmunit4}\qquad
S_j=\mat{ccccc} \!\!\Lambda_j&\!\!\Theta_j&\!\!\!\!-r_2\Lambda_j\!\!\!&\!\!\dots & \!\!\!-r_{m_j-1}\Lambda_j\!\!\!\!\\
0&\!\!\ddots&\ddots&\!\!\ddots&\vdots\\[-1ex] \vdots&\!\!\ddots&\ddots&\!\!\ddots&\!\!-r_2\Lambda_j\!\!\\[-0.5ex]
\vdots&&\!\!\ddots&\!\!\ddots&\Theta_j\\ 0 & \!\!\cdots&\cdots&\!\!0&\Lambda_j\rix,\quad
J_j=\sigma_jR_j\otimes \Sigma_2.
\end{equation}
where $|\lambda_j|=1$, $\mbox{\rm Im}(\lambda_j)>0$, $m_j\in\mathbb N$, $\sigma_j=\pm 1$, and
\begin{equation}\label{phitheta}
\Lambda_j=\mat{cc}\mbox{\rm Re}(\lambda_j)&\mbox{\rm Im}(\lambda_j)\\ -\mbox{\rm Im}(\lambda_j)&\mbox{\rm Re}(\lambda_j)\rix,\quad
\Theta_j=\mat{cc}\mbox{\rm Im}(\lambda_j)&-\mbox{\rm Re}(\lambda_j)\\\mbox{\rm Re}(\lambda_j)&\mbox{\rm Im}(\lambda_j)\rix.
\end{equation}
Moreover, $r_k=0$ for odd $k$ and the parameters $r_k$
for even $k$ are real and uniquely determined by the recursive formula
\begin{equation}\label{bruuhl}
r_2=\frac{1}{2},\quad
r_{k}=\frac{1}{2}\left(\sum_{\nu=1}^{\frac{k}{2}-1}r_{2\cdot \nu}r_{2\cdot(\frac{k}{2}-\nu)}\right),
\quad 4\leq k\leq m_j;
\end{equation}
\item[v)] blocks associated with a quadruplet $(\alpha_j\pm i\beta_j,\frac{1}{\alpha_j\pm i\beta_j})$
of nonreal, nonunimodular eigenvalues:
\begin{equation}\label{symmunit5}\qquad
S_j=\mat{cc}\mathcal J_{m_j}(\alpha_j,\beta_j)& 0\\ 0 & \big(\mathcal J_{m_j}(\alpha_j,\beta_j)\big)^{-\top}\rix,
\quad J_j=\mat{cc} 0&I_{2m_j}\\ - I_{2m_j}&0\rix,
\end{equation}
where $\alpha_j\in\mathbb R$, $\beta_j>0$, $\alpha_j^2+\beta_j^2>1$, and $m_j\in\mathbb N$.
\end{itemize}
The values $\lambda_j$, $m_j$, $\sigma_j$, $\alpha_j$, and $\beta_j$ may be different for different values of $j$.
Moreover, the form~\eqref{symmunit} is unique up to the permutation of blocks.
\end{theorem}

The signs $\sigma_j\in\{1,-1\}$ appearing in the blocks in~\eqref{symmunit1} and~\eqref{symmunit4}
are additional invariants. For a fixed unimodular eigenvalue $\lambda\in\mathbb C$, the collection
of all signs associated with the corresponding blocks is called the \emph{sign characteristic} of $\lambda$, see \cite{GohLR05} for more information about
the sign characteristic of matrices that have symmetry structures with respect
to an indefinite inner product.

An important tool in the remainder of the paper will be \emph{$J$-symplectic reflectors},
sometimes also called \emph{$J$-symplectic transvections}.

\begin{definition}\label{def:sr}
Let $u\in \mathbb F^{2n}\setminus\{0\}$ and $c\in\mathbb F\setminus\{0\}$.
Then a matrix of the form $I_{2n}+c uu^\top J$ is called a \emph{$J$-symplectic reflector}.
\end{definition}

Since $J^\top=-J$ and since for any $u\in\mathbb F^{2n}$ we have $u^\top Ju=0$, it follows easily from
\[
(I_{2n}+c uu^\top J)^\top J(I_{2n}+c uu^\top J)=
J-c Juu^\top J+c Juu^\top J+c^2 Juu^\top Juu^\top J=J
\]
that a $J$-symplectic reflector is symplectic. Obviously it is a rank-one
perturbation of the identity, and moreover we have $(I_{2n}+cuu^\top J)^{-1}=I_{2n}-cuu^\top J$,
i.e., the inverse of a symplectic reflector is again a symplectic reflector.
Furthermore, by scaling $u$ with the reciprocal of
a square root of $c$, we may assume that $c=1$ in the case $\mathbb F=\mathbb C$.
Similarly, we may assume $c\in\{1,-1\}$ in the case $\mathbb F=\mathbb R$.

It was shown in \cite[Lemma 7.1]{MMRR5} that if $S,\widetilde S\in\mathbb C^{2n,2n}$ are two $J$-symplectic matrices satisfying
$\operatorname{rank}(S-\widetilde S)=1$, then $\widetilde S$ can be obtained from $S$ by left multiplication with a
$J$-symplectic reflector, i.e., there exists a nonzero vector $u\in\mathbb C^{2n}$ such that
$\widetilde S=(I_{2n}+uu^\top J)S$. If both $S$ and $\widetilde S$ are real, then analogously it can be
shown that there exists a nonzero $u\in\mathbb R^{2n}$ such that $\widetilde S=(I_{2n}\pm uu^\top J)S$.
We highlight that despite of its multiplicative nature, a symplectic rank-one perturbation of the form $\widetilde S=(I_{2n}\pm uu^\top J)S$
can also be interpreted as a particular additive perturbation of the form $\widetilde S=S\pm uu^\top JS$. Therefore, the following
observation on additive rank-one perturbations also applies to structure-preserving rank-one perturbation of symplectic matrices.

\begin{remark}\label{rem.geochangeone}\rm
If $A$ is any matrix and $B$ is a second matrix of the same size that has rank one, then it is obvious that the ranks of $A$ and $A+B$ can differ at most by one. This simple, but important
observation leads to a crucial fact that we will use frequently in the following:
\emph{the geometric multiplicity of any particular eigenvalue
of a square matrix can change by at most one after a rank-one perturbation}.
\end{remark}

The fact that a symplectic rank-one perturbation of a $J$-symplectic matrix $S$ has the form  $\widetilde S=(I_{2n}\pm uu^\top J)S$
for some $u\in\mathbb F^{2n}$ allows a parametrization of the set of structure-preserving rank-one perturbations.
Based on this observation the following result characterizing the change in the Jordan structure of $J$-symplectic
matrices under generic structure-preserving perturbations was obtained
in \cite{MMRR5}. Note that the statement refers to complex eigenvalues even in the case
$\mathbb F=\mathbb R$.

\begin{theorem}\label{thm.rankonepert}
Let $J\in\mathbb F^{2n, 2n}$ be skew-symmetric and invertible and let $S\in\mathbb F^{2n, 2n}$
be $J$-symplectic. Then, there exists a generic set $\Omega\subseteq\mathbb F^{2n}$
such that for all $u\in\Omega$ and all $\sigma\in\{1,-1\}$ the matrix
\[
\widetilde S=(I_{2n}+\sigma uu^\top J)\cdot S.
\]
satisfies the following statements:
\begin{itemize}
\item[${\rm (1)}$] Let $\lambda\in\mathbb C$ be any eigenvalue of $S$, and let it have geometric multiplicity $g$
and partial multiplicities $n_1\geq n_2\geq \cdots\geq n_g$. If $g>1$ then the following statements hold:
\begin{itemize}
\item[${\rm (1a)}$] If $\lambda\neq 1,-1$ or if $\lambda\in\{1,-1\}$ and $n_1$ is even,
then $\lambda$ is an eigenvalue of $\widetilde S$ having precisely the partial multiplicities
$n_{2}\geq\cdots\geq n_{g}$.
\item[${\rm (1b)}$] If $\lambda\in\{1,-1\}$ and $n_1$ is odd,
then then $\lambda$ is an eigenvalue of $\widetilde S$ having precisely the partial multiplicities
$n_{2}+1\geq n_{3}\geq\cdots\geq n_{g}$.
\end{itemize}
If, on the other hand, $g=1$, then $\lambda$ is not an eigenvalue of $\widetilde S$.
\item[${\rm (2)}$] If $\mu\in\mathbb C$ is an eigenvalue of $\widetilde S$, but not of $S$,
then $\mu$ is a simple eigenvalue of $\widetilde S$.
\end{itemize}
\end{theorem}

\proof
The proof for the case $\mathbb F=\mathbb C$ was
given in \cite{MMRR5}. If on the other hand $\mathbb F=\mathbb R$, then
we can first apply the result for the complex case to the real matrix
$S$ which gives a generic set $\widetilde\Omega\subseteq\mathbb C^n$
such that (1) and (2) both hold. The statement for the case $\mathbb F=\mathbb R$ then follows
immediately from the following lemma that implies that the set
$\Omega:=\widetilde\Omega\cap\mathbb R^n$ is generic as well.
\eproof

\begin{lemma}[\cite{MMRR4}]\label{lem.c2r}
Let $\mathcal A\subseteq \mathbb C^n$ be a proper algebraic set. Then
$\mathcal A\cap\mathbb R^n$ is a proper algebraic set in $\mathbb R^n$.
\end{lemma}

In the following we aim to extend Theorem~\ref{thm.rankonepert} to the case $k>1$. The main idea is to successively apply
Theorem~\ref{thm.rankonepert} and for this, we need to generalize \cite[Lemma 7.1]{MMRR5}, i.e.,
we have to answer the following question: if $S,\check S$ are two symplectic matrices satisfying rank$(\check S-S)=k$, do there exist
$k$ symplectic reflectors $S_1,\dots,S_k$ such that $\check S=S_k\cdot\ldots\cdot S_1\cdot S$? Equivalently, we can consider
the matrix $S^{-1}\check S$ that satisfies rank$(S^{-1}\check S-I_{2n})=k$
and ask if it can be written as a product of $k$ symplectic reflectors.
This question will be investigated in detail in the next section.

\section{Products of symplectic reflectors}\label{sec:fac}

It is well known that the group of symplectic matrices is generated by
symplectic reflectors and an elementary proof of this fact can be found in \cite{BueR17}.
As pointed out by \cite{Dop15}, a constructive method
to factorize any real symplectic $2n\times 2n$ matrix into a product of at most
$4n$ symplectic reflectors can be obtained combining the procedures from
\cite{Meh79} and \cite{FlaMZ91}. For the purpose of this paper, it is however necessary
to write a symplectic matrix $S$ satisfying rank$(S-I_{2n})=k$ as a product of exactly
$k$ symplectic reflectors. As revealed in the introduction this is not always possible.

\begin{example}\label{ex.gpwrong}\rm
Let $S=-I_2$ and $J=\Sigma_2$. Then $S$ is $J$-symplectic with $\operatorname{rank}(S-I_2)=2$, but it cannot be written as a product
of two symplectic reflectors. Indeed, assume that this is possible, i.e., there exist $u,v\in\mathbb C^2$
such that $S=(I_2+uu^\top J)(I_2+vv^\top J)$. Since the inverse of $I_2+uu^\top J$ is $I_2-uu^\top J$ which
is again a symplectic reflector, we obtain that
\begin{equation}\label{ex.5.10.23}
(I_2-uu^\top J)\mat{cc}-1&0\\ 0&-1\rix=(I_2+vv^\top J)\mat{cc}1&0\\ 0&1\rix.
\end{equation}
On the left hand side, we have a rank-one perturbation of $-I_2$. Since by Remark~\ref{rem.geochangeone} the geometric multiplicity
of eigenvalues can change at most by one in this case, it follows
that $-1$ is still an eigenvalue of the matrix on the left hand side. But since this matrix is
$J$-symplectic, it follows from the canonical form in Theorem~\ref{cbilkunithm} that its second
eigenvalue must be equal to $-1$ as well, i.e., the spectrum of the matrix on the left hand side
consists of the single element $-1$. A similar argument shows that the matrix on the right hand
side can only have the eigenvalue $+1$ which contradicts the equality of both matrices. This is in line with Theorem~\ref{thm.rankonepert} which states
that generically the left hand side of~\eqref{ex.5.10.23} will be similar to
a $2\times 2$ Jordan block associated with the eigenvalue $-1$, while
the right hand side of~\eqref{ex.5.10.23} will be similar to
a $2\times 2$ Jordan block associated with the eigenvalue $1$.

It is possible to write $S$ as a product of three symplectic reflectors though.
One can easily verify that for
\[
u_1=\mat{c}i\\ 0\rix,\;u_2=\mat{c}1\\ 2i\rix,\;u_3=\mat{c}1+i\\ 2i\rix
\]
we have $(I_2+u_3u_3^\top J)(I_2+u_2u_2^\top J)(I_2+u_1u_1^\top J)S=I_2$ and hence
\[
S=(I_2-u_1u_1^\top J)(I_2-u_2u_2^\top J)(I_2-u_3u_3^\top J).
\]
\end{example}

The question under which conditions a symplectic matrix $S$ satisfying rank$(S-I_{2n})=k$ can be written as a product of exactly 
$k$ symplectic reflectors is answered in the following theorem. This result highlights that indeed blocks like the one
considered in Example~\ref{ex.gpwrong} are problematic and lead to exceptional cases.

\begin{theorem}\label{thm:fac}
Let $S\in\mathbb F^{2n,2n}$ be $J$-symplectic such that $\operatorname{rank}(S-I_{2n})=k$.
Then $S$ is a product of $k$ $J$-symplectic reflectors, i.e., there exist
$u_1,\dots,u_k\in\mathbb F^{2n}$ and $c_1,\dots,c_k\in\{+1,-1\}$ (in the case $\mathbb F=\mathbb C$ all $c_j$ can
be chosen to be $+1$) such that
\[
S=\prod_{j=1}^k(I_{2n}+c_ju_ju_j^\top J),
\]
unless $S$ is similar to a matrix of the form $I_{2(n-p)}\oplus (-I_{2p})$
for some $p=1,\dots,n$.
In the latter case, $S$ is a product of $k+1$ $J$-symplectic reflectors.
\end{theorem}

\begin{remark}\rm
Theorem~\ref{thm:fac} is a special case of a more general result by Dieudonn\'{e} \cite{Die55} that was reproved in 
\cite[Theorem 2.1.11]{OMe78}. It is shown that in the case $F\neq\mathbb F_2$ a symplectic matrix satisfying $\operatorname{rank}(S-I_{2n})=k$ can be written as a product of $k$ symplectic reflectors if and only if it is not \emph{hyperbolic}. A symplectic $S$ is called
hyperbolic if for all vectors $x\in F^{2n}$ one has $x^\top SJx=0$ (with $J$ as in~\eqref{eq:J}) and it is shown that this is 
equivalent to $S$ being similar to a matrix of the form $I_{2(n-p)}\oplus (-I_{2p})$. (Also the case $F=\mathbb F_2$ is
considered in \cite{Die55}, but then in addition to the hyperbolicity
there are more exceptional cases for which $k+1$ reflectors are needed instead of $k$. Interestingly, the list of exceptions presented 
in \cite{Die55} turned out to be incomplete, but this was corrected later in \cite{Cal76}.)
\end{remark}

The proofs given in \cite{Die55} or \cite{OMe78} for arbitrary fields are quite involved. We will present an independent proof using different arguments (via canonical forms)
for the case $F=\mathbb C$. In the author's opinion, this proof gives new insight in what is happening from the viewpoint of low rank perturbation theory and perfectly explains why
the exceptional condition in Theorem~\ref{thm:fac} is exactly how it is.

\medskip

{\it Proof} (of Theorem~\ref{thm:fac}). Let $\mathbb F=\mathbb C$. Without loss of generality we may assume
that $S$ and $J$ are in the canonical form from Theorem~\eqref{cbilkunithm}.
As noted above, blocks
of the form~\eqref{cbilkuni2} associated with the eigenvalue $-1$ and having
size $m_j=1$, that is,
\begin{equation}\label{exceptblock}
S_j=\mat{cc} -1&0\\ 0&-1\rix,\quad J_j=\mat{cc}0&1\\ -1&0\rix,
\end{equation}
are problematic blocks. We will therefore consider two main cases.

{\it Case (1)}: \emph{The canonical form for $(S,J)$ does not have any blocks of the form~\eqref{exceptblock}}.
We will prove the assertion by induction on $k$. The case $k=0$ is trivial, thus
assume that $k>0$. Since $\operatorname{rank}(S-I_{2n})=k$ it follows that the geometric multiplicity
of the eigenvalue $\lambda_0=1$ is $2n-k$. We now consider
several subcases and in any subcase, we will show that there exists a vector $u\in\mathbb C^{2n}$ such that $\widehat S:=(I_{2n}+uuJ)S$
satisfies $\operatorname{rank}(\widehat S-I_{2n})=k-1$ and still does not have any blocks of the form~\eqref{exceptblock}.
The assertion of the theorem then follows by applying the
induction hypothesis to the new matrix $\widehat S$.

{\it Subcase (1a)}: \emph{The canonical form for $(S,J)$ contains a block of the form~\eqref{cbilkuni3}}.
By Theorem~\ref{cbilkunithm} we may assume that $S$ and $J$ are in the form
\[
S=\mat{cc}\mathcal J_{m}(\lambda)&0\\ 0& \big(\mathcal J_{m}(\lambda)\big)^{-T}\rix\oplus\widetilde S,\quad
J=\mat{cc} 0 & I_{m}\\ -I_{m} & 0\rix\oplus\widetilde J
\]
for some $\lambda\in\mathbb C$ satisfying $\lambda^2\neq 1$ and for some
$m$. Then choosing $u=\alpha e_m+\beta e_{2m}$, we obtain that
\[
\widehat S:=(I_{2n}+uu^\top J)S=S_1\oplus\widetilde S,
\]
where
\begin{eqnarray*}
S_1&=&\mat{cccc}I_{m-1}&0&0&0\\ 0&1-\alpha\beta&0&\alpha^2\\ 0&0&I_{m-1}&0\\ 0&-\beta^2&0&1+\alpha\beta\rix\cdot
\mat{cccc}\mathcal J_{m-1}(\lambda)&e_{m-1}&0&0\\ 0&\lambda&0&0\\
0&0&\big(\mathcal J_{m-1}(\lambda)\big)^{-T}&0\\ 0&0&\ast&\lambda^{-1}\rix\\
&=&\mat{cccc}\mathcal J_{m-1}(\lambda) & e_{m-1} & 0&0\\
0&(1-\alpha\beta)\lambda & \ast& \alpha^2\lambda^{-1}\\
0&0&\big(\mathcal J_{m-1}(\lambda)\big)^{-T}& 0\\ 0&-\beta^2\lambda & \ast & (1+\alpha\beta)\lambda^{-1}\rix,
\end{eqnarray*}
where the symbol $\ast$ stands for block entries that are not of interest. It is easily seen that the
eigenvalues of $S_1$ are given by $\lambda$ and $\lambda^{-1}$ (both with algebraic
multiplicity $m-1$) and the eigenvalues of the two by two matrix
\[
M:=\mat{cc}(1-\alpha\beta)\lambda&\alpha^2\lambda^{-1}\\ -\beta^2\lambda&(1+\alpha\beta)\lambda^{-1}\rix
\]
whose characteristic polynomial is given by
$p(t)=t^2-\big((1-\alpha\beta)\lambda+(1+\alpha\beta)\lambda^{-1}\big)t+1$. Thus, for the choice
\[
\alpha=\frac{2-(\lambda+\lambda^{-1})}{\lambda^{-1}-\lambda}\quad\mbox{and}\quad\beta=1
\]
we obtain $(1-\alpha\beta)\lambda+(1+\alpha\beta)\lambda^{-1}=2$ and hence
$M$ has a double eigenvalue at one and the Jordan canonical form $\mathcal J_2(1)$.
Clearly, then also $S_1$ has the Jordan block $\mathcal J_2(1)$ in its Jordan canonical form
as well as $\widehat S$ (in addition to other Jordan blocks associated with $\lambda_0=1$
from the part $\widetilde S$ which has remained unchanged). Thus, we obtain
that $\operatorname{rank}(\widehat S-I_{2n})=k-1$. By construction, no block of the form~\eqref{exceptblock} 
has been generated in $\widehat S$.

{\it Subcase (1b)}: \emph{The canonical form for $(S,J)$ contains a block of the form~\eqref{cbilkuni1} associated with the
eigenvalue $\lambda=1$.} In that case, we may assume
by Theorem~\ref{cbilkunithm} that $S$ and $J$ are in the form
\[
S=T(1,1,r_2,\dots,r_{m-1})\oplus \widetilde S,\quad J=\Sigma_m\oplus\widetilde J.
\]
for some  $m=2\ell>0$, where $r_2,\dots,r_{m-1}$ are as in Theorem~\ref{cbilkunithm}. Then taking $u=i^\ell e_\ell$, we obtain that
\[
\widehat S:=(I_{2n}+uu^\top J)S=S_1\oplus \widetilde S,
\]
where
\[
S_1=\mat{cccccc}1&&&&&\\[-1ex] &\ddots&&&&\\ &&1&-1&&\\ &&&1&&\\[-1ex] &&&&\ddots&\\ &&&&&1\rix\cdot
\mat{cccccc}1&1&\ast&\ast&\ast&\ast\\[-1ex] &1&\ddots&\ast&\ast&\ast\\[-1ex]
&&\ddots&\ddots&\ast&\ast\\[-1ex]&&&\ddots&\ddots&\ast\\ &&&&1&1\\ &&&&&1\rix=
\mat{cccccc}\ddots&\ddots&\ast&\ast&\ast&\ast\\ &1&1&\ast&\ast&\ast\\[-1ex]
&&\ddots&0&\ast&\ast\\[-1ex]&&&\ddots&1&\ast\\[-1ex] &&&&1&\ddots\\[-1ex] &&&&&\ddots\rix
\]
Then one immediately sees that $\operatorname{rank}(S_1-I_m)=m-2$ which implies that
the geometric multiplicity of the eigenvalue $\lambda_0=1$ of $S_1$ (and thus $\widehat S$)
has increased by one. It follows that $\widehat S$ satisfies
$\operatorname{rank}(\widehat S-I_{2n})=k-1$. As in the previous case no block of the form~\eqref{exceptblock} has been generated in $\widehat S$.

{\it Subcase (1c)}: \emph{The canonical form for $(S,J)$ contains a block of the form~\eqref{cbilkuni2} associated with the
eigenvalue $\lambda=1$ with $m_j>1$.} In that case, we may assume by Theorem~\ref{cbilkunithm} that $S$ and $J$ are in the form
\[
S=\mat{cc}\mathcal J_{m}(1)&0\\ 0& \big(\mathcal J_{m}(1)\big)^{-T}\rix\oplus\widetilde S,\quad
J=\mat{cc} 0 & I_{m}\\ -I_{m} & 0\rix\oplus\widetilde J,
\]
where for simplicity we write $m$ instead of $m_j$. Then taking $u=\frac{1}{\sqrt{2}}(e_{m-1}+e_{2m})$, we obtain that
\[
\widehat S:=(I_{2n}+uu^\top J)S=S_1\oplus \widetilde S,
\]
where, using the abbreviation $z=\mat{ccc}(-1)^{m-1}&\dots&(-1)^2\rix=\mat{cccccc}1&-1&\dots&1&-1&1\rix$, we have
\begin{eqnarray*}
S_1&=&\mat{ccc|ccc}I_{m-2}&0&0&0&0&0\\ 0&1&-\frac{1}{2}&0&\frac{1}{2}&0\\ 0&0&1&0&0&0\\\hline
0&0&0&I_{m-2}&0&0\\ 0&0&0&0&1&0\\ 0&0&-\frac{1}{2}&0&\frac{1}{2}&1\rix\cdot
\mat{ccc|ccc}\mathcal J_{m-2}(1) & e_{m-2} & 0&0 & 0&0\\ 0&1&1&0&0&0\\ 0&0&1&0&0&0\\\hline\rule{0mm}{4.5mm}
0&0&0&\big(\mathcal J_{m-2}(1)\big)^{-T}& 0&0\\ 0&0&0&-z&1 & 0\\
0&0&0&z&-1&1\rix\\
&=&\mat{ccc|ccc}\mathcal J_{m-2}(1) & e_{m-2} & 0&0 & 0&0\\ 0&1&\frac{1}{2}&-\frac{1}{2}z&\frac{1}{2}&0\\ 0&0&1&0&0&0\\\hline\rule{0mm}{4.5mm}
0&0&0&\big(\mathcal J_{m-2}(1)\big)^{-T}& 0&0\\ 0&0&0&-z&1 & 0\\
0&0&-\frac{1}{2}&\frac{1}{2}z&-\frac{1}{2}&1\rix
\end{eqnarray*}
One easily checks that rank$(S_1-I_{2m})=2m-3$. (Indeed, one can observe that in the corresponding block representation as
above for $S_1-I_{2m}$, the block rows $2,3,5$ and $6$ span a subspace of dimension two, while in addition, there is a zero row in the $4$th block row.)
Thus, the geometric multiplicity of the eigenvalue $1$ of $S_1$ has
increased by one. Furthermore, $1$ is still the only eigenvalue of $S_1$
and hence no block of the form~\eqref{exceptblock}
has been generated in $\widehat S$.

{\it Subcase (1d)}: \emph{The canonical form for $(S,J)$ contains a block of the form~\eqref{cbilkuni1} associated with the
eigenvalue $\lambda=-1$.} In that case, we may assume by Theorem~\ref{cbilkunithm} that $S$ and $J$ are in the form
\[
S=T(-1,1,r_2,\dots,r_{m-1})\oplus \widetilde S,\quad J=\Sigma_m\oplus\widetilde J.
\]
for some  $m=2\ell>0$, where $r_2,\dots,r_{m-1}$ are as in Theorem~\ref{cbilkunithm}. We will first consider the case
$\ell\neq2$, i.e. $m\neq4$. Then taking $u=2i^\ell e_{\ell+1}$, we obtain that
\[
\widehat S:=(I_{2n}+uu^\top J)S=S_1\oplus \widetilde S,
\]
where
\[
S_1=\mat{cccccc}1&&&&&\\[-1ex] &\ddots&&&&\\ &&1&&&\\ &&4&1&&\\[-1ex] &&&&\ddots&\\ &&&&&1\rix\cdot
\mat{ccccc}-1&1&\ast&\ast&\ast\\[-1ex] &-1&\ddots&\ast&\ast\\[-1ex]
&&\ddots&\ddots&\ast\\ &&&-1&1\\ &&&&-1\rix=
\mat{ccc}T&\ast&\ast\\ 0&\mat{cc}-1&1\\-4&3\rix&\ast\\ 0&0&T\rix,
\]
with $T=T(-1,1,r_2,\dots,r_{\ell-2})$. Since the $2\times 2$ block in the middle has a double eigenvalue at $\lambda=1$,
it follows that the geometric multiplicity of the eigenvalue $1$ has increased.
It remains to show that no block of the form~\eqref{exceptblock}
has been generated in $\widehat S$.
In the case $\ell=1$ this is clear as $S_1$ only has the eigenvalue one. If on the other hand $\ell>2$, i.e., $m>4$, then
$S_1$ has the eigenvalue $-1$ with algebraic multiplicity $m-2$. Furthermore,
$T+I_{\ell-1}$ has rank $\ell-2$ which implies that $S_1+I_m$ has at least rank $m-2$, i.e., the geometric
multiplicity of the eigenvalue $-1$ of $S_1$ is at most two. But then $S_1$ cannot have a block of the
form~\eqref{exceptblock}, because such a block already corresponds to
geometric multiplicity two, but only to algebraic multiplicity two
instead of $m-2$.

In the case $\ell=2$, i.e., $m=4$, we will need a slightly different perturbation. In fact, in this case we would have
rank$(S+I_4)=2$ for the perturbation above and this implies that an unwanted block of the form~\eqref{exceptblock} had been generated. Instead, we take $u=e_1+2e_3$ which,
using that by Theorem~\ref{cbilkunithm} we have $r_2=\frac{1}{2}$ and $r_3=0$, gives
\[
S_1=\mat{cccc}1&2&0&1\\ 0&1&0&0\\ 0&4&1&2\\ 0&0&0&1\rix\cdot\mat{cccc}-1&1&\frac{1}{2}&0\\ 0&-1&1&\frac{1}{2}\\ 0&0&-1&1\\ 0&0&0&-1\rix
=\mat{cccc}-1&-1&\frac{5}{2}&0\\ 0&-1&1&\frac{1}{2}\\ 0&-4&3&1\\ 0&0&0&-1\rix
\]
and it is straightforward to check that this matrix has the Jordan canonical form $\mathcal J_2(1)\oplus\mathcal J_2(-1)$.

{\it Subcase (1e)}: \emph{The canonical form for $(S,J)$ contains a block of the form~\eqref{cbilkuni2} associated with the
eigenvalue $\lambda=-1$ having the size $m_j>1$.} In that case, we may assume by Theorem~\ref{cbilkunithm} that $S$ and $J$ are in the form
\[
S=\mat{cc}\mathcal J_{m}(-1)&0\\ 0& \big(\mathcal J_{m}(-1)\big)^{-T}\rix\oplus\widetilde S,\quad
J=\mat{cc} 0 & I_{m}\\ -I_{m} & 0\rix\oplus\widetilde J,
\]
where for simplicity we write $m$ instead of $m_j$, again. Then taking $u=e_{m}-2e_{2m-1}$, we obtain that
\[
\widehat S:=(I_{2n}+uu^\top J)S=S_1\oplus \widetilde S,
\]
where, using the abbreviation $w:=\mat{ccc}1&\dots&1\rix\in\mathbb F^{1,m-2}$, we have
\begin{eqnarray*}
S_1&=&\mat{ccc|ccc}I_{m-2}&0&0&0&0&0\\ 0&1&0&0&0&0\\ 0&2&1&0&0&1\\\hline
0&0&0&I_{m-2}&0&0\\ 0&-4&0&0&1&-2\\ 0&0&0&0&0&1\rix\cdot
\mat{ccc|ccc}\mathcal J_{m-2}(-1) & e_{m-2} & 0&0 & 0&0\\ 0&-1&1&0&0&0\\ 0&0&-1&0&0&0\\\hline\rule{0mm}{4.5mm}
0&0&0&\big(\mathcal J_{m-2}(-1)\big)^{-T}& 0&0\\ 0&0&0&-w&-1 & 0\\
0&0&0&-w&-1&-1\rix\\
&=&\mat{ccc|ccc}\mathcal J_{m-2}(-1) & e_{m-2} & 0&0 & 0&0\\ 0&-1&1&0&0&0\\ 0&-2&1&-w&-1&-1\\\hline\rule{0mm}{4.5mm}
0&0&0&\big(\mathcal J_{m-2}(-1)\big)^{-T}& 0&0\\ 0&4&-4&w&1&2\\
0&0&0&-w&-1&-1\rix.
\end{eqnarray*}
Thus, the characteristic polynomial $p$ of $S_1$ is given by
\[
p=(t+1)^{2m-4}\cdot\left|\begin{array}{cccc}-1-t&1&0&0\\ -2&1-t&-1&-1\\ 4&-4&1-t&2\\ 0&0&-1&-1-t\end{array}\right|
=(t+1)^{2m-2}(t-1)^2
\]
and consequently, the geometric multiplicity of the eigenvalue one has increased.
It remains to show that no block of the form~\eqref{exceptblock}
has been generated in $\widehat S$.
One easily checks that the matrix $S_1+I_{2m}$ has at least rank $2m-2$. (Indeed, in the corresponding block representation as
above, the block rows $1$ and $4$ yield $2m-5$ linearly independent rows already, and adding the rows $2,3,5$ and $6$ to this
set of rows increases the dimension of the corresponding subspace  by three.) Using the same argument as in the first part
of subcase (1d), it follows that $S_1$ cannot have a block of the
form~\eqref{exceptblock}.

We highlight that the (not mutually exclusive) subcases (1a)--(1e) cover all possible cases. Indeed, if none of these cases would apply,
then the canonical form of $(S,J)$ can only consist of blocks of the form~\eqref{cbilkuni2}
associated with the eigenvalue $1$ and size $m_j=1$. But then $S$ is the identity
which is excluded by the assumption $k>0$.

{\it Case (2)}: \emph{The canonical form for $(S,J)$ does have blocks of the form~\eqref{exceptblock}}. In this case, we proceed by considering the following two subcases.

{\it Subcase (2a)}: \emph{$S$ is not similar to a matrix of the form
$I_{2(n-p)}\oplus(-I_{2p})$ for some $p\in\{1,\dots,n\}.$} This means that
the canonical form for $(S,J)$ contains at least one block as considered in the subcases
(1a)--(1e). In other words, $S$ either has an eigenvalue $\lambda$ satisfying $\lambda^2\neq 1$ or at least
one of the eigenvalues $1$ and $-1$ is not semisimple (unless $S=I_{2n}$).

If $S$ has at least one block as in the subcases (1a)--(1c), then
we show that we can proceed using the same strategy as in
Case (1), i.e., increasing the geometric multiplicity of the eigenvalue one in each step, while simultaneously
reducing the number of blocks of the form~\eqref{exceptblock} in each step.
To this end, we may assume without loss of generality that $S$ and $J$ are in the form
\[
S=\check S\oplus\mat{cc}-1&0\\ 0&-1\rix\oplus\widetilde S\quad\mbox{and}\quad
J=\check J\oplus\mat{cc}0&1\\ -1&0\rix\oplus\widetilde J,
\]
where $(\check S,\check J)$ is exactly one block as considered in the subcases (1a)--(1c). Thus, we consider only
one of the blocks of the form~\eqref{exceptblock} at a time,
and we gather the remaining ones (if any) together with all other blocks in the matrices $\widetilde S$ and $\widetilde J$.
Now let $u_1$ be such that
the rank-one perturbation $(I+u_1u_1^\top\check J)\check S=S_1$, where $S_1$ is the perturbed matrix from the
corresponding subcase (1a)-(1c) that has the geometric multiplicity of the eigenvalue $1$ increased by one.
Then setting $u=\mat{ccc}u_1^\top&1&0\rix^\top\in\mathbb C^{2n}$, we obtain that
\[
(I_{2n}+uu^\top J)S=S_2\oplus\widetilde S
\]
with
\[
S_2=\mat{ccc}I+u_1u_1^\top\check J&0&u_1\\ u_1^\top\check J&1&1\\ 0&0&1\rix\cdot\mat{ccc}\check S&0&0\\ 0&-1&0\\ 0&0&-1\rix
=\mat{ccc}S_1&0&-u_1\\ u_1^\top\check J\check S&-1&-1\\ 0&0&-1\rix.
\]
If $\check S$ is as considered in the subcases (1a)--(1c), i.e., its eigenvalues are either $1$ or a pair
$\big(\lambda_0,\frac{1}{\lambda_0}\big)$ of eigenvalues with $\lambda_0^2\neq 1$, then clearly $S_2$ and $S_1$ have the same Jordan structure
on the eigenvalues different from $-1$. This follows immediately from $\operatorname{rank}(S_2-\lambda I)^\ell=\operatorname{rank}(S_1-\lambda I)^\ell+2$
for all $\lambda\neq -1$ and all $\ell\in\mathbb N$. On the other hand, $S_1$ does not have the eigenvalue $-1$ and so we have
$\operatorname{rank}(S_2+I)=\operatorname{rank}S_1+1$ which implies that $S_2$ has the eigenvalue $-1$ with algebraic multiplicity two,
but geometric multiplicity one. This means that one of the blocks of the form~\eqref{exceptblock} has transformed into a block of the form~\eqref{cbilkuni1} having size $m_j=2$.

If on the other hand, the canonical form of $(S,J)$ consists only of
blocks as in the subcases (1d) and (1e) (possibly besides blocks of the
form~\eqref{cbilkuni2} with size $m_j=1$), then we proceed with executing
exactly one step
as described in subcase (1d) or (1e), respectively. Observe that this
generates a block of the form~\eqref{cbilkuni1} associated with the
eigenvalue $1$ and having size $m_j=2$. Since the presence of this block assures that we have not yet finished our induction (the induction
only stops if we have arrived at the identity matrix), we can then use this block to
reduce the number of blocks of the form~\eqref{exceptblock} as described in the previous
paragraph.

So far, we have described how to remove one block of the
form~\eqref{exceptblock}. Observe that this results in producing
one block of the form~\eqref{cbilkuni1} associated with the eigenvalue $-1$ having size $m_j=2$. Thus, we can continue to remove the next block of
the form~\eqref{exceptblock} (if any) by the procedure explained in the
previous paragraph, i.e., by applying a rank-one perturbation
as described in case (1d) which leaves a block of the
form~\eqref{cbilkuni1} having size $m_j=2$ associated with the eigenvalue
$+1$ which can then be used to eliminate the next block of the
form~\eqref{exceptblock}. Consequently, after a finite number of steps, all
blocks of the form~\eqref{exceptblock} will have disappeared and we
can proceed as in Case~(1).

{\it Subcase (2b)}: \emph{$S$ is similar to a matrix of the form
$I_{2(n-p)}\oplus(-I_{2p})$ for some $p\in\{1,\dots,n\}.$} We first show that $S$ cannot be written as a
product of $k$ $J$-symplectic reflectors. Assume that this would be
possible, i.e., there exist $k$ vectors $u_1,\dots,u_k$ such that
\[
S=\prod_{j=1}^k(I_{2n}+u_ju_j^\top J),\quad\mbox{i.e.}\quad
\left(\prod_{j=1}^k(I_{2n}-u_ju_j^\top J)\right)\cdot S=I_{2n}.
\]
Since the geometric multiplicity of each eigenvalue can change
by at most one under a rank-one perturbation, it follows from
rank$(S-I_{2n})=k$ that the matrix
\[
\left(\prod_{j=1}^\ell(I_{2n}-u_ju_j^\top J)\right)\cdot S
\]
must have the eigenvalue $1$ with geometric multiplicity $2n-k+\ell$
for $\ell=1,\dots k$.
Since the matrix $S$ is similar to the matrix $I_{2(n-p)}\oplus(-I_{2p})$,
we see that $S$ has the eigenvalue $1$ with algebraic and geometric
multiplicity $2(n-p)$ and the eigenvalue $-1$ with algebraic and geometric
multiplicity $2p$ which implies $k=2p$. Considering the matrix
$(I_{2n}-u_1u_1^\top J)S$, it follows from the fact that the geometric
multiplicity of the eigenvalue $-1$ can change by at most one that it
can either be $2p-1$ or $2p$ or $2p+1$. In all
three cases, the algebraic multiplicity of the eigenvalue $-1$ is at least
$2p$. (Indeed, in the first case, we cannot have algebraic multiplicity
$2p-1$, because the algebraic multiplicity of the eigenvalue $-1$ always
has to be even. This is in line with Theorem~\ref{thm.rankonepert}, which states that generically
one block of the form~\eqref{cbilkuni2} will change to a block of the
form~\eqref{cbilkuni1}.) But this means that the algebraic multiplicity
(and hence also the geometric multiplicity) of the eigenvalue $1$ can be
at most $2n-2p$ which is a contradiction to the observation that it
should be equal to $2n-2p+1$ after the first rank-one perturbation.
Thus, $S$ cannot be written as a product of $k$ $J$-symplectic reflectors.

On the other hand, we may assume by Theorem~\ref{cbilkunithm} that $(S,J)$ is in the canonical form
\[
S=I_{2(n-p)}\oplus(-I_{2p}),\quad J=\oplus_{j=1}^n\mat{cc}0&1\\ -1&0\rix.
\]
Then choosing $u=e_{2(n-p)+1}$, we obtain that
\[
\check S=(I_{2n}+uu^\top J)S=I_{2(n-p)}\oplus\mat{cc}-1&-1\\ 0&-1\rix\oplus(-I_{2(p-1)}).
\]
This particular rank-one perturbation has left the eigenvalue $1$ with
geometric multiplicity $2(n-p)=2n-k$, but also has generated one block
of the form~\eqref{cbilkuni1} with size $m_j=2$ for the eigenvalue $-1$.
We know from Case (2a) (or Case (1) if $p=1$, respectively) that
$\check S$ can be written as a product of $k$ $J$-symplectic reflectors.
Consequently, $S$ can be written as a product of $k+1$ $J$-symplectic
reflectors. \eproof

\medskip

\begin{remark}\rm
The proof of Theorem~\ref{thm:fac} perfectly explains why only symplectic
matrices similar to a matrix $I_{2(n-p)}\oplus(-I_{2p})$ for some
$p\in\{1,\dots,n\}$ are exceptional. We have shown that any block
$S_j$ of the form~\eqref{cbilkuni2} with size $m_j>1$, or
of the form~\eqref{cbilkuni1} or~\eqref{cbilkuni3} can be written
as a product of rank$(S_j-I)$ $J$-symplectic reflectors. The problem are
blocks of the form~\eqref{exceptblock}, because in that case a rank-one
perturbation cannot reduce the algebraic multiplicity of the eigenvalue.
This is crucial in the case of a matrix similar to
$I_{2(n-p)}\oplus(-I_{2p})$, $p\in\{1,\dots,n\}$, since in that case the
geometric multiplicity of the eigenvalue $1$ cannot be increased. If,
however, other blocks are present in the matrix, then those can be used
to both increase the geometric multiplicity of the eigenvalue $1$ while
simultaneously reducing the geometric multiplicity of the eigenvalue $-1$.

The same strategy can be used to extend the proof to the case
$\mathbb F=\mathbb R$. This, however, would require much more work, because
in this case the canonical form is more complicated. Since the result in
that case is already known in the literature and the purpose of gaining additional insight was already taken
care of by proving the case $\mathbb F=\mathbb C$, we omit the proof in the
real case and refer to \cite{Die55} and \cite{OMe78} instead.
\end{remark}

As a consequence of Theorem~\ref{thm:fac} we can show that sufficiently
small symplectic rank-$k$ perturbations can always be written as products
of $k$ symplectic reflectors.

\begin{corollary}\label{cor:smallpert}
Let $J\in\mathbb F^{2n, 2n}$ be skew-symmetric and invertible and let $S\in\mathbb F^{2n, 2n}$ be
$J$-symplectic. Then there exists $\delta>0$ such that for each $k\in\{1,\dots,2n\}$ and each $J$-symplectic
matrix $\widetilde S\in\mathbb F^{2n, 2n}$ satisfying $\operatorname{rank}(\widetilde S-S)=k$
and $\|\widetilde S-S\|_2<\delta$ there exist vectors $u_1,\dots,u_k\in\mathbb F^{2n}\setminus\{0\}$ and values $c_1,\dots,c_k\in\{1,-1\}$
such that
\[
\widetilde S=\left(\prod_{j=1}^k(I_{2n}+c_ju_ju_j^\top J)\right)\cdot S.
\]
The values $c_j$ can all be chosen to be equal to $+1$ in the case
$\mathbb F=\mathbb C$.
\end{corollary}

\proof
Observe that we have $\operatorname{rank}(\widetilde SS^{-1}-I_{2n})=k$ and
$\|\widetilde SS^{-1}-I_{2n}\|<\delta\|S\|^{-1}$. Since eigenvalues depend continuously on the
entries of the given matrix, it follows that for sufficiently small $\delta>0$ all eigenvalues
of $\widetilde SS^{-1}$ are contained in an open ball around $\lambda_0=1$ with radius $2$,
i.e., $\widetilde SS^{-1}$ does not have the eigenvalue $-1$. Then the statement follows immediately
by applying Theorem~\ref{thm:fac}.
\eproof

\begin{remark}\label{rem.para}\rm
Theorem~\ref{thm:fac} allows us to parameterize a large class of symplectic perturbations of rank at
most $k$ by a polynomial map. Setting
\begin{equation}\label{PU}
P(U)=\left(\prod_{j=1}^k(I_{2n}+u_ju_j^\top J)-I_{2n}\right)S,\quad \mbox{where }U=\mat{ccc}u_1&\dots&u_k\rix\in\mathbb C^{2n,k}
\end{equation}
we obtain that $S+P(U)=\prod_{j=1}^k(I_{2n}+u_ju_j^\top J)S$ is a symplectic perturbation of $S$
such that $P(U)$ has rank at most $k$. Obviously, $P(U)$ is a polynomial map in the $2nk$ complex entries of $U$ as
well as in the $4nk$ real and imaginary parts of the entries of $U$. We highlight that
it is not guaranteed that $P(U)$ has rank exactly $k$ for a given $U\in\mathbb C^{2n,k}$.
In fact, it may well be that the rank of $P(U)$ is less than $k$. However, it follows from
\cite[Lemma 2.1]{MMRR2} that the set of all $U$ for which $P(U)$ has rank $k$ is a generic
subset of $\mathbb C^{2n,k}$.
\end{remark}

The exceptional matrices from Theorem~\ref{thm:fac} cannot be
parameterized as in~\eqref{PU}. Fortunately, the set of these matrices
is a rare set in the following topological sense.

\begin{corollary}\label{cor.rare}
Let $J\in\mathbb F^{2n,2n}$ be skew-symmetric and invertible. Then
the set $\mathcal S$ of $J$-symplectic matrices that are similar to a matrix of the form $I_{2(n-p)}\oplus (-I_{2p})$ for some $p=1,\dots,n$ 
is a closed nowhere dense subset of measure zero of the set of $J$-symplectic matrices.
\end{corollary}

\proof
The closedness easily follows by a continuity argument. Indeed, if $(S_k)_{k\in\mathbb N}$
is a sequence in $\mathcal S$ with limit $S$ then $S$ is a symplectic matrix with
spectrum contained in $\{+1,-1\}$ due to the closedness of the symplectic groups and
the continuity of eigenvalues. Due to the lower semi-continuity of the rank it follows
that there exist $g_1,g_{-1}\in\mathbb N$ such that the geometric multiplicity of the eigenvalue $+1$ is at least $g_1$ and the one of the eigenvalue $-1$ at least $g_{-1}$
for $S$ and almost all $S_k$. Since the matrices $S_k$ are from $\mathcal S$ it follows
that $g_{-1}=2p$ and $g_1=2(n-p)$ and since these two values add up to $2n$, the only
situation that $S$ can attain these two geometric multiplicities at $+1$ and $-1$ is
when $S$ is similar to the matrix $I_{2(n-p)}\oplus (-I_{2p})$.

Next, we show that $\mathcal S$ has empty interior. But this is clear, because
if $S$ is any symplectic matrix similar to a matrix of the form $I_{2(n-p)}\oplus (-I_{2p})$
then there exist symplectic matrices arbitrarily close to $S$ that have eigenvalues
different from $\pm 1$, so $S$ is not an interior point of $\mathcal S$.

The remaining parts follow immediately from the fact that $\mathcal S$ is contained in the set of $J$-symplectic matrices $S$
having the eigenvalue $-1$ which is a hypersurface (i.e. subvariety of codimension one) in the variety of symplectic matrices.
(Indeed, the set of $J$-symplectic matrices is defined by the equation $S^\top JS-J=0$ that
consists of equations that are polynomials in the entries of $S$. The mentioned subset of $J$-symplectic matrices is then a proper subset 
that is obtained by adding one more polynomial equation, namely $\det(S+I)=0$.) 
It is well known that hypersurfaces of algebraic varieties are sets of measure zero in the given variety.
\eproof

\medskip

Thus, we conclude that the set of symplectic matrices
$S$ that can be written as a product of $\operatorname{rank}(S-I_{2n})$ symplectic
reflectors is \emph{generic} in the sense of Corollary~\ref{cor.rare}. Within
this set, we can now investigate the generic behavior of symplectic matrices under
low-rank perturbations using the parametrization from Remark~\ref{rem.para}. This is
done in the following section.

\section{Structure-preserving low-rank perturbation of symplectic matrices}

In this section we will state the main result of the paper on the behavior of the Jordan structure of symplectic matrices
under symplectic rank-$k$ perturbations.
For its proof we need the following auxiliary result which is a minor adaptation of \cite[Lemma 3.1]{MR17}.

\begin{theorem}\label{local}
Let $A\in\mathbb C^{n,n}$ and let the Jordan canonical form of $A$ be
given by
$$
\mathcal{J}_{n_1}(\lambda) \oplus \cdots \oplus \mathcal{J}_{n_g}(\lambda)\oplus\tilde{\mathcal{J}},
$$
with $n_1 \geq \cdots \geq n_g$ and $\lambda\not\in\sigma(\tilde{\mathcal{J}})$.
Furthermore, let $P\in\mathbb C^{n,n}[t_1,\dots,t_r]$ be a matrix whose entries are polynomials
in the independent indeterminate variables $t_1,\dots,t_r$. Assume that for all $u=(u_1,\dots,u_r)\in\mathbb F^r$
we have
\begin{enumerate}
\item[{\rm (i)}] $\operatorname{rank}P(u)\leq k $;
\item[{\rm (ii)}] the algebraic multiplicity of any eigenvalue of $A+P(u)$ is always a multiple of
$\nu\in\mathbb N\setminus\{0\}$.
\end{enumerate}
Then the following statements hold:
\begin{enumerate}
\item
For each $(u_1,\dots,u_r)\in\mathbb F^r$ there exist integers $\eta_1 \geq \cdots \geq \eta_\ell$
such that
\begin{enumerate}
\item[{\rm (a)}] the Jordan canonical form of $A+P(u_1,\dots,u_r)$ is given by
$$
\mathcal{J}_{\eta_1}(\lambda) \oplus \cdots \oplus \mathcal{J}_{\eta_\ell}(\lambda)\oplus\check{\mathcal{J}},
$$
where  $\lambda\not\in\sigma(\check{\mathcal{J}})$,
\item[{\rm (b)}] $(\eta_1, \ldots , \eta_\ell)$ dominates $(n_{k +1}, \ldots , n_g)$; that is, we have
$\ell\geq g-k $ and $\eta_j \geq n_{j+k }$ for $j=1, \ldots , g-k $.
\end{enumerate}
\item
Assume that for all $u=(u_1,\dots,u_r)\in\mathbb F^r$ the algebraic multiplicity $a_{u}$ of $\lambda$
as an eigenvalue of $A+P(u)$ satisfies $a_{u}\geq a$ for some $a\in\mathbb{N}$.
If there exists $u_0\in\mathbb F^r$ such that $a_{u_0} =a$, then the set
$$
\Omega =\{ u\in\mathbb F^r\mid a_{u}=a\}
$$
is a generic set.
\item
Assume that for any $\varepsilon>0$ there exists $u_0\in\mathbb F^r$ with $\|u_0\|<\varepsilon$
such that the Jordan form of $A+P(u_0)$ is described by
\begin{itemize}
\item[{\rm (a)}] $
\mathcal{J}_{n_{k +1}}(\lambda) \oplus \cdots \oplus \mathcal{J}_{n_g}(\lambda)\oplus\check{\mathcal{J}}, \qquad \lambda\not\in\sigma(\check{\mathcal{J}}),
$
\item[{\rm (b)}] all eigenvalues that are not eigenvalues of $A$ have multiplicity $\nu$ precisely.
\end{itemize}
Then there exists a generic set $\Omega\subseteq\mathbb F^r$ such that the Jordan canonical form of
$A+P(u)$ is described by {\rm (a)} and {\rm (b)} for all $u\in\Omega$.
\end{enumerate}
\end{theorem}

\proof
The proof of Theorem~\ref{local} for the case $\mathbb F=\mathbb R$ was given in
\cite{MR17}. For the case $\mathbb F=\mathbb C$ the proof is completely identical
and therefore omitted.
\eproof

\begin{remark}\rm
\begin{enumerate}
\item The first part of Theorem~\ref{local} classifies restrictions to the Jordan structure
of matrices after a general (not necessarily generic) rank-$k$ perturbation. Part~2
states that the minimal possible algebraic multiplicity that can be obtained under
a general perturbation of a specific rank is also the generic one. Thus, to prove that a particular
algebraic multiplicity is the one that is obtained after a generic perturbation, it
suffices to prove that it is a lower bound for all possible algebraic multiplicities
under general perturbations and to then find one particular example, where this
algebraic multiplicity is attained. The third part deals with the generic
multiplicity of newly generated eigenvalues. The value $\nu$ will be equal to one in
most applications (i.e., newly generated eigenvalues will generically be simple)
and that is how we will apply the result in this paper.
The result is stated in a more general way, because there are some instances when the
result is applied with $\nu=2$, e.g. when the considered set is the set of skew-Hamiltonian matrices. (It is well known
that all eigenvalues of skew-Hamiltonian matrices appear with even algebraic multiplicity.)
\item Observe that both the entries of the matrix $A$ and the image of the
function $P$ are assumed to be complex. Since $A$ and $P$ are fixed, the result can also be
applied in the case that $A$ is real and $P$ is real-valued, because $\mathbb R$ is
contained in $\mathbb C$. In contrast, observe that the more
general symbol $\mathbb F\in\{\mathbb C,\mathbb R\}$ was used for the parameters
to be inserted in the polynomial function $P$. In this paper, we will only use
complex parameters and thus apply the result for the case $\mathbb F=\mathbb C$ only.
The reason why the statement of Theorem~\ref{local} is formulated in a more general
way is the potential application in cases where complex conjugation appears.
We explain the issue for the theory of rank-one perturbations of Hermitian
$n\times n$ matrices. Here, the corresponding structure-preserving
rank-one perturbations take the form $P(u)=uu^\ast$ with $u\in\mathbb C^n$ and it
turns out that $P$ is \emph{not} a polynomial map in the complex entries of $u$ as
complex conjugation is involved. However, considering the real and imaginary parts
of each entry of $u$, we find that $P(u)$ is a polynomial in the
$2n$ real variables that are obtained in that way. In this particular case, it is
necessary to consider real parameters even in the case when complex matrices
are considered.
\end{enumerate}
\end{remark}

The second auxiliary result that we will need is \cite[Lemma 2.2]{Bat14b}, see also \cite[Lemma 2.1]{BatMRR16}.

\begin{lemma}\label{notalg}
Let $\mathcal B\subseteq\mathbb F^\ell$ not be contained in any proper algebraic subset of $\mathbb F^\ell$.
Then, $\mathcal B\times\mathbb F^k$ is not contained in any proper algebraic subset of $\,\mathbb F^\ell\times\mathbb F^k$.
\end{lemma}

With the factorization of rank-$k$ perturbations of symplectic matrices from Section~\ref{sec:fac}, we
are now able to describe the generic behavior of symplectic matrices under structure-preserving
rank-$k$ perturbations.

\begin{theorem}\label{mainresult}
Let $J\in\mathbb F^{2n, 2n}$ be skew-symmetric and invertible and let $S\in\mathbb F^{2n, 2n}$
be $J$-symplectic. Furthermore let $k\in\mathbb N\setminus\{0\}$ and for $j=1,\dots,k$ let $c_j=1$ if $\mathbb F=\mathbb C$, and
$c_j\in\{1,-1\}$ if $\mathbb F=\mathbb R$. Then, there exists a generic set $\Omega_k\subseteq\mathbb F^{n, k}$
such that for all $U\in\Omega_k$ the matrix
\[
\widetilde S=\left(\prod_{j=1}^k(I_{2n}+c_ju_ju_j^\top J)\right)\cdot S,\quad c_j\in\{1,-1\}
\]
satisfies the following statements:
\begin{itemize}
\item[${\rm (1)}$] Let $\lambda\in\mathbb C$ be any eigenvalue of $S$, and let it have geometric multiplicity $g$
and partial multiplicities $n_1\geq n_2\geq \cdots\geq n_g$. If $k<g$ then the following statements hold:
\begin{itemize}
\item[${\rm (1a)}$] If $\lambda\neq 1,-1$ or if $\lambda\in\{1,-1\}$ and $n_1+\dots+n_k$ is even,
then $\lambda$ is an eigenvalue of $\widetilde S$ having precisely the partial multiplicities
$n_{k+1}\geq n_{k+2}\geq\cdots\geq n_{g}$.
\item[${\rm (1b)}$] If $\lambda\in\{1,-1\}$ and $n_1+\dots+n_k$ is odd,
then $\lambda$ is an eigenvalue of $\widetilde S$ having precisely the partial multiplicities
$n_{k+1}+1\geq n_{k+2}\geq\cdots\geq n_{g}$.
\end{itemize}
If, on the other hand, $k\geq g$, then $\lambda$ is not an eigenvalue of $\widetilde S$.
\item[${\rm (2)}$] If $\mu\in\mathbb C$ is an eigenvalue of $\widetilde S$, but not of $S$,
then $\mu$ is a simple eigenvalue of $\widetilde S$.
\end{itemize}
\end{theorem}

\proof
We will use the same tactic as in the proof of the main theorem of \cite{BatMRR16} by successively applying Theorem~\ref{thm.rankonepert}.
We will focus on one particular eigenvalue $\lambda\in\mathbb C$ of $S$ having the partial multiplicities
$n_1\geq\cdots\geq n_g$. We will prove the result for the case
$\mathbb F=\mathbb C$ only. The result for the case $\mathbb F=\mathbb R$
then follows analogously as in the proof of Theorem~\ref{thm.rankonepert}
by applying Lemma~\ref{lem.c2r}.

Note that in the case $\mathbb F=\mathbb C$ we may assume without loss of generality
For the remainder of the proof, we consider various cases.

{\it Case (1)}: $\lambda\neq 1,-1$. Note that by Theorem~\ref{local}, it is
sufficient to construct examples for rank-$k$ perturbations having sufficiently
small norms such that $n_{k+1}\geq n_{k+2}\geq\cdots\geq n_{g}$ are the partial multiplicities
at $\lambda$ as an eigenvalue of the perturbed matrix and such that all newly created eigenvalues are simple.
This example is easily obtained by successively applying $k$ generic rank-one perturbations to our matrix $S$ and using the
result for $k=1$ that is already available.

To be more precise, let $\varepsilon>0$.
Applying Theorem~\ref{thm.rankonepert}, there exists a vector $u_1\in\mathbb C^n\setminus\{0\}$ with norm $\|u_1\|<\frac{\varepsilon}{\sqrt{n}}$
such that the matrix $S_1:=(I_{2n}+u_1u_1^\top J)S$ has the eigenvalue $\lambda$ with partial multiplicities $n_2\geq\cdots\geq n_g$  
and such that all eigenvalues of $S_1$ that are not eigenvalues of $S$ are simple.
Then applying Theorem~\ref{thm.rankonepert} again, but this time to the matrix $S_1$, there exists a vector $u_2\in\mathbb C^n\setminus\{0\}$
with $\|u_2\|<\frac{\varepsilon}{\sqrt{n}}$ such that the matrix
\[
S_2:=(I_{2n}+u_2u_2^\top J)S_1=(I_{2n}+u_2u_2^\top J)(I_{2n}+u_1u_1^\top J)S
\]
has the eigenvalue $\lambda$ with partial multiplicities $n_3\geq\cdots\geq n_g$ (if $g=2$,
then $\lambda$ is not an eigenvalue of $S_1$) and all eigenvalues that are not eigenvalues of $S$ are simple. (To see the latter
point, a little more subtle argument is necessary, because the actual result on generic rank-one perturbations will only claim that
the eigenvalues of $S_2$ that are no eigenvalues of $S_1$ are simple. But the eigenvalues that were eigenvalues of $S_1$, but not
of $S$ were all simple, and hence the result on generic rank-one perturbations tells us that all these simple eigenvalues of $S_1$
will not be eigenvalues of $S_2$.) Then using induction, we finally find $k$ vectors $u_1,\dots,u_k\in\mathbb C^n\setminus\{0\}$
that give the desired example satisfying $\|\mat{ccc}u_1&\dots&u_k\rix\|<\varepsilon$.

{\it Case (2)}: {\it $\lambda\in\{1,-1\}$ and $n_1+\cdots+n_k$ is even}. Again, it is sufficient to construct
one example for a rank-$k$ perturbation such that $n_{k+1}\geq n_{k+2}\geq\cdots\geq n_{g}$ are the partial multiplicities
at $\lambda$ as an eigenvalue of the perturbed matrix and such that all newly created eigenvalues are simple.
This case follows mainly as in the previous case. To see this, let us assume that we have already applied $\ell-1<k$ generic rank-one
perturbation given by the vectors $u_1,\dots,u_{\ell-1}$, so that the perturbed matrix
\[
S_{\ell-1}=\left(\prod_{j=1}^{\ell-1}(I_{2n}+u_ju_j^\top J)\right)\cdot S
\]
has the partial multiplicities $n_{\ell}\geq n_{\ell+1}\geq\cdots\geq n_{g}$ at $\lambda$ with $n_1+\cdots+n_{\ell-1}$
(and hence also $n_\ell+\cdots+n_g$) being even. If $n_\ell$ is even, then we can proceed as in the previous Case~1.
Note that then the sum $n_{\ell+1}+\cdots+n_g$ is again even.
If on the other hand $n_\ell$ is odd, then by the fact that the number of partial multiplicities of a fixed odd number must
be even and by the supposed parities of the sums $n_1+\cdots+n_{\ell-1}$ and $n_\ell+\cdots+n_g$ it follows that $n_{\ell+1}=n_\ell$
and $\ell+1\leq k$. Now a generic rank-one perturbation of $S_{\ell-1}$ given by a vector $u_\ell\in\mathbb C^{2n}\setminus\{0\}$ will produce the
partial multiplicities $n_{\ell+1}+1\geq n_{\ell+2}\geq\cdots\geq n_{g}$ at $\lambda$. Since $\ell+1\leq k$, we can apply another
 (the $\ell+1$st) generic rank-one
perturbation on $S_\ell=(I_{2n}+u_{\ell}u_\ell^\top J)S_{\ell-1}$ given by a nonzero vector $u_{\ell+1}$. Since $n_{\ell+1}$ is now even, the
perturbed matrix
\[
S_{\ell+1}=(I_{2n}+u_{\ell+1}u_{\ell+1}^\top J)S_{\ell}=\left(\prod_{j=1}^{\ell+1}(I_{2n}+u_ju_j^\top J)\right)\cdot S
\]
will now have the partial multiplicities $n_{\ell+2}\geq\cdots\geq n_{g}$ at $\lambda$. Furthermore, the sum $n_{\ell+2}+\cdots+n_{g}$ is
again even so that we can proceed by induction. (In the argument we concentrated on the multiplicities of $\lambda$. Clearly, the
newly generated eigenvalues will again all be simple and the particular example
for the perturbation can be chosen to be of arbitrarily small norm.)

{\it Case (3)}: {\it $\lambda\in\{1,-1\}$ and $n_1+\cdots+n_k$ is odd}. In this case, it is not possible that the partial multiplicities
at $\lambda$ of a rank-$k$ perturbation of $S$ are $n_{k+1}\geq\cdots\geq n_g$. Indeed, from the canonical form, we know
that the algebraic multiplicity of the eigenvalue $\lambda$ must be even, and hence with $n_1+\cdots+n_k$ also $n_{k+1}+\cdots+n_g$ is odd.
Thus, the minimal possible algebraic multiplicity $a$ of $\lambda$ as an eigenvalue of a symplectic matrix $S_k$ obtained by a structure-preserving
rank-$k$ perturbation of $S$ is $a:=n_{k+1}+\cdots+n_g+1$. Let $\varepsilon>0$.
By Theorem~\ref{local} it is then sufficient to construct one particular example of
a rank-$k$ perturbation with norm smaller than $\varepsilon$
that results in a perturbed matrix with algebraic multiplicity $a$ in order to show that generically, all rank-$k$
perturbations of $S$ will have the eigenvalue $\lambda$ with algebraic multiplicity $a$. This example can be constructed as in the previous case $(2)$.
First, the fact that $n_1+\cdots+n_k$ is odd implies that $n_k=n_{k+1}$ is odd and hence $n_1+\cdots+n_{k-1}$ is even.
(In particular, this implies that $g\geq k+1$.) Thus, as in Case~(2), we can construct vectors $u_1,\dots,u_{k-1}$ such that the matrix
\[
S_{k-1}=\left(\prod_{j=1}^{k-1}(I_{2n}+u_ju_j^\top J)\right)\cdot S
\]
has the partial multiplicities $n_k\geq n_{k+1}\geq\cdots\geq n_g$ at $\lambda$. Then applying a further (generic) rank-one perturbation
on $S_{k-1}$ given by the vector $u_k$, it follows that the matrix $S_k=(I+u_ku_k^\top J)S_{k-1}$ will have the partial multiplicities
$n_{k+1}+1\geq n_{k+2}\geq\cdots\geq n_g$ at $\lambda$, and hence the algebraic multiplicity of $\lambda$ is $a=n_{k+1}+\cdots+n_g+1$. (Again the vectors $u_1,\dots,u_k$ can
be chosen in such a way that the norm of the rank-$k$ perturbation is less than
$\varepsilon$.)

It remains to show that the perturbed matrix
\[
S_k=\left(\prod_{j=1}^{k}(I_{2n}+u_ju_j^\top J)\right)\cdot S
\]
will have the partial multiplicities $n_{k+1}+1\geq n_{k+2}\geq\cdots\geq n_g$ at $\lambda$ generically with respect to $U=\mat{ccc}u_1&\dots&u_k\rix$.
To this end, we recall that by Theorem~\ref{local} the list $(\eta_1,\dots,\eta_\ell)$ of partial multiplicities of $S_k$ at $\lambda$ must dominate
the list $(n_{k+1},n_{k+2},\dots,n_g)$. Since the generic algebraic multiplicity $a$ differs only by one from the sum $n_{k-1}+\cdots+n_g$, this means
that either exactly one partial multiplicity from the list $(n_{k+1},\dots,n_g)$ must have increased by one (i.e., we have $\ell=g-k$ and $\eta_i=n_{k+i}+1$
for some $i$ and $\eta_j=n_{k+j}$ for all $j\neq i$) or an additional partial multiplicity of size one was added to the list (i.e., we have
$\ell=g-k+1$, $\eta_i=n_{k+i}$ for $i=1,\dots,g-k$ and $\eta_{g-k+1}=1$). Simultaneously, the list $(\eta_1,\dots,\eta_\ell)$ must satisfy the requirements
set by the canonical form of symplectic matrices. This rules out many lists per se - in fact there is only one particular situation in which we
have two different lists of potential partial multiplicities at $\lambda$. To illustrate this, let us present two examples.

If we consider the list $(n_{k+1},\dots,n_g)=(7,7,7,4,3,3,1,1)$, then the only list that corresponds to an algebraic multiplicity increased by one,
that dominates the list $(7,7,7,4,3,3,1,1)$ and is a possible list of partial multiplicities of $\lambda$ for a symplectic matrix is the list
$(\eta_1,\dots,\eta_\ell)=(8,7,7,4,3,3,1,1)$. Indeed, adding an extra partial multiplicity of size $1$ or adding $1$ to any of the partial multiplicities
lower than $7$ will result in a list of partial multiplicities that has an odd-sized partial multiplicity occurring an odd number of times - a fact
that is not allowed by the canonical form of symplectic matrices.

If on the other hand, we have $(n_{k+1},\dots,n_g)=(7,7,7,6,3,3,1,1)$, then there are two different dominating lists corresponding to an algebraic multiplicity
increased by one that are also compatible with a list of partial multiplicities of $\lambda$ of a symplectic
matrix, namely $(8,7,7,6,3,3,1,1)$ and $(7,7,7,7,3,3,1,1)$. In this case, we can add one either to the largest partial multiplicity $7$ or to
the second largest partial multiplicity $6$ and obtain in both cases lists where all odd partial multiplicities occur an even number of times.

Let us now formalize the observation from the previous two examples. To this end, assume that the partial multiplicity $n_{k+1}$ appears
$s$ times in the list $(n_{k+1},\dots,n_g)$, that is, we have $n_{k+1}=\cdots=n_{k+s}>n_{k+s+1}$. Since we assumed that
the sum $n_{k+1}+\cdots+n_g$ is odd, it is necessary that $s$ is odd itself. We now distinguish two subcases.

{\it Subcase (3a)}: $n_k>n_{k+s+1}+1$. We have already proved that generically the algebraic multiplicity of a symplectic matrix $S_k$
obtained by a rank-$k$ perturbation from $S$ is $a=n_{k+1}+\cdots+n_g+1$. Since $n_k>n_{k+s+1}+1$, we find that the only list
of partial multiplicities of $S_k$ at $\lambda$ that corresponds to the algebraic multiplicity $a$, that dominates $(n_{k+1},\dots,n_g)$
and that is compatible with the list of partial multiplicities at $\lambda$ of a symplectic matrix is the list
$(n_{k+1}+1,n_{k+2},\dots,n_g)$. Indeed, the partial multiplicity $n_k$ now appears the even number $s-1$ of times in the list. If any other
partial multiplicity $n_i\neq n_k$ would have been increased by one instead, then the corresponding list would have the
partial multiplicity $n_k$ the odd number $s$ of times in it which is not compatible with the Jordan structure of any symplectic matrix. This finishes the
proof in this case.

{\it Subcase (3b)}: $n_k=n_{k+s+1}+1$. As in the previous subcase, we can assume that generically the perturbed symplectic matrix $S_k$
has algebraic multiplicity $a=n_{k+1}+\cdots+n_g+1$. Now, there are two possible lists of partial multiplicities of $S_k$ at $\lambda$
that correspond to the algebraic multiplicity $a$, that dominate $(n_{k+1},\dots,n_g)$ and that are compatible with the list of partial
multiplicities at $\lambda$ of a symplectic matrix, namely $(n_{k+1}+1,n_{k+2},\dots,n_g)$ and
$(n_{k+1},\dots,n_{k+s},n_{k+s+1}+1,\dots,n_g)$. (In the first list, the partial multiplicity $n_k$ appears the even number $s-1$
of times, and in the second list $s+1$ times.) It remains to show that generically, we do have the first of the two lists.
Assume that this is not the case. Then there exists a set $\mathcal B\subseteq\mathbb C^{n,k}$ which is not contained in an
proper algebraic subset of $\mathbb C^{n,k}$ such that for all $(u_1,\dots,u_k)\in\mathcal B$ we have that the matrix
\[
S_k=\left(\prod_{j=1}^{k}(I_{2n}+u_ju_j^\top J)\right)\cdot S
\]
has at $\lambda$ the partial multiplicities $(n_{k+1},\dots,n_{k+s},n_{k+s+1}+1,\dots,n_g)$. Now let us apply a further symplectic
rank-one perturbation on $S_k$. It follows that for any $u_{k+1}\in\mathbb C^n$ the algebraic multiplicity of
\[
S_{k+1}=(I_{2n}+u_{k+1}u_{k+1}^\top J)S_k
\]
can decrease at most by $n_{k+1}$. Thus, we obtain that for all $(u_1,\dots,u_{k+1})\in\mathcal B\times\mathbb C^n$ the
algebraic multiplicity of
\[
S_{k+1}=(I_{2n}+u_{k+1}u_{k+1}^\top J)S_k=\left(\prod_{j=1}^{k+1}(I_{2n}+u_ju_j^\top J)\right)\cdot S
\]
at $\lambda$ is at least $a-n_{k+1}=n_{k+2}+\cdots+n_g+1$. As the set $\mathcal B\times\mathbb C^n$ is not contained in an
algebraic set by Lemma~\ref{notalg} we arrive at a contradiction, because by the already proved Case~2 we find that generically
the algebraic multiplicity of $S_{k+1}$ is $n_{k+2}+\cdots+n_g$. This proves the genericity of the partial multiplicities
$n_{k+1}+1\geq n_{k+2}\geq\cdots\geq n_g$ of $S_{k+1}$ at $\lambda$. The simplicity of the newly generated eigenvalues is
proved analogously to the previous cases.
\eproof

\section*{Conclusions}

In this paper a missing result in the theory of low-rank perturbations of matrices has been provided. This
result describes the generic behavior of the Jordan structure of real and complex symplectic matrices that are subject to
low-rank structure-preserving transformations.

\section*{Acknowledgement}

The author would like to thank Julio Moro and Fredy Sosa for many fruitful discussions on the topic of the paper.

\end{document}